\documentclass[11pt]{amsart}
\usepackage[mathlines, displaymath]{lineno}
\usepackage{amscd,amssymb, amsmath,  setspace, wasysym, yfonts, amsthm, mathtools, dirtytalk}
\usepackage{graphicx}
\usepackage[all]{xy}
\usepackage{tikz}
\usetikzlibrary{matrix}

\usepackage{url}
\usepackage{hyperref}

\theoremstyle{plain}
\newtheorem{theorem}{Theorem}[section]
\newtheorem{prop}[theorem]{Proposition}

\newtheorem{cor}[theorem]{Corollary}

\newtheorem{lemma}[theorem]{Lemma}

\theoremstyle{definition}

\newtheorem{defn}[theorem]{Definition}

\newtheorem{ex}[theorem]{Example}
\newtheorem*{ex*}{Example}

\newcommand\sO{{\mathcal O}}
\newcommand\sK{{\mathcal K}}

\newcommand\sH{{\mathcal H}}

\newcommand\sP{{\mathcal P}}
\newcommand\sF{{\mathcal F}}
\newcommand\sG{{\mathcal G}}

\newcommand\cc{{\mathbb{C}}}

\DeclareMathOperator{\Pic}{Pic}

\DeclarePairedDelimiterX\set[1]\lbrace\rbrace{#1}

\title[Reconstruction theorems in the supported case]
{Reconstruction theorems in the supported case}

\author{Luigi Lombardi}
\address{Department of Mathematics  \\ Università degli Studi di  Milano\\Via Cesare Saldini 50, 20133 Milano, Italy }
 \email{\url{luigi.lombardi@unimi.it}}
\begin{document}

\begin{abstract}
We show that  any equivalence of bounded derived categories of coherent sheaves on a smooth projective complex variety supported in a closed 
algebraic subset preserves
the dimension of the support in  two cases: \emph{(i)} the restriction of the (anti)canonical bundle to the support is ample; \emph{(ii)}
the supports are irreducible and the equivalence sends a skyscraper sheaf of a closed point to a skyscraper sheaf of a closed point. 
Moreover, in the first case the equivalence recovers the set of closed points of the support up to homeomorphism. 

\end{abstract}

\maketitle 

\section{Introduction}

We study instances regarding  the invariance of the dimension of the support
 under exact equivalences of bounded derived categories of coherent sheaves   supported in a closed algebraic subset.

Let $X$ and $Y$ be  smooth projective complex varieties and let $\mathrm{  D}^b(\mathrm{  Coh}(X))$ and \sloppy $\mathrm{  D}^b({\rm Coh}(Y))$ 
be the bounded derived categories of coherent sheaves on $X$ and $Y$, respectively. 
Under certain conditions $\mathrm{  D}^b({\rm Coh}(X))$ reconstructs the underlying variety.
In this direction we recall the celebrated Bondal--Orlov's reconstruction theorem: If the (anti)canonical bundle of $X$ is ample, then 
any exact equivalence $F \colon \mathrm{  D}^b({\rm Coh}(X)) \to \mathrm{  D}^b({\rm Coh}(Y))$ induces an isomorphism of varieties $X\simeq Y$ (\emph{cf}. 
\cite{BO}; see also \cite{F,Ito} for the non-proper case and  \cite{Ba,IM, SS} for the Gorenstein case). 
By weakening the hypotheses one obtains coarser notions of reconstruction. For instance, if the (anti)canonical bundle is big, or 
if the equivalence $F$   sends a skyscraper sheaf of a closed point to a skyscraper sheaf of a closed point, 
then $X$ and $Y$ are $K$-equivalent, and in particular birational 
(\emph{cf}. \cite[Theorem 1.4]{Ka}, \cite[Theorem 1.1]{CL}, \cite[Lemma 7.3]{T} and \cite[Theorem 1.1]{LLZ}).

Given algebraic  closed subsets $Z\subset X$ and $W\subset Y$, in this paper 
we consider the bounded derived categories $\mathrm{  D}_Z^b({\rm Coh}(X))$ and $\mathrm{  D}_W^b({\rm Coh}(Y))$  of complexes of coherent sheaves with 
 cohomologies supported in $Z$ and $W$, respectively.  
By the existence of the Serre functor, it is easy to check that any exact equivalence 
$F \colon \mathrm{  D}_Z^b(\mathrm{  Coh}(X)) \to \mathrm{  D}_W^b(\mathrm{  Coh}(Y))$  
preserves the dimensions of $X$ and 
$Y$ (Corollary \ref{cor:eqdim}). The following theorem implies the invariance of the dimension of the support under a positivity condition
on the canonical bundle $\omega_X$.

\begin{theorem}\label{thm:intr1}
If  $\mathrm{  D}_Z^b({\rm Coh}(X))\simeq \mathrm{  D}_W^b(\mathrm{  Coh}(Y))$ 
 and $\omega_X|_Z$, or $\omega_X^{-1}|_Z$, is ample, then the sets of closed points  $Z(\cc)$ and $W(\cc)$ are homeomorphic. 
\end{theorem}

Instances of equivalences of type $\mathrm{  D}_Z^b(\mathrm{  Coh}(X))\simeq \mathrm{  D}_W^b(\mathrm{  Coh}(Y))$  as in Theorem \ref{thm:intr1} 
appear in \cite[Lemma 4.6]{T} where it is showed that any equivalence 
$\mathrm{  D}^b(\mathrm{  Coh}(X)) \simeq \mathrm{  D}^b(\mathrm{  Coh}(Y))$ induces  isomorphisms of 
linear series $\big|\omega_X^{\otimes k} \big| \simeq \big |\omega_Y^{\otimes k}  \big |$ (for all
$k\in \mathbb Z$) and  equivalences  
\begin{equation}\label{eq:toda}
\mathrm{  D}_{\mathrm{  Supp}(E)}^b(\mathrm{  Coh}(X)) \;  \simeq \; \mathrm{  D}_{\mathrm{  Supp} ( E')}^b(\mathrm{  Coh}(Y))
\end{equation}
for any pair of corresponding divisors 
$E \in \big |\omega_X^{\otimes k} \big|$  
and $E' \in \big |\omega_Y^{\otimes k} \big |$ (see also Example \eqref{ex:R}).

In the second main theorem of the paper we consider an
equivalence
 $\mathrm{  D}_Z^b(\mathrm{  Coh}(X))\simeq \mathrm{  D}_W^b(\mathrm{  Coh}(Y))$ 
that sends a skyscraper sheaf of a closed point to a skyscraper sheaf of a closed point as in \cite{T}. 
We prove that if the supports are irreducible then   they have
the same dimension. In fact something stronger holds, and 
in general we speculate the invariance of the dimension of the support under an arbitrary equivalence.
 \begin{theorem}\label{thm:intr2}
 Suppose $Z\subset X$ and $W\subset Y$ are algebraic closed irreducible subsets and let
   $F \colon \mathrm{  D}_Z^b(\mathrm{  Coh}(X)) \to  \mathrm{  D}_W^b(\mathrm{  Coh}(Y))$ be an equivalence. 
   If $F$ sends   a skyscraper sheaf  of a closed point in $Z$ to a skyscraper sheaf  of a closed point in $W$, then  there exist Zariski open dense subsets $U\subset Z$
and $V\subset W$ such that the sets of closed points  $U(\cc)$ and $V(\cc)$ are homeomorphic.  
Moreover, if  $\dim W=1$, then the sets of closed points $Z(\cc)$ and $W(\cc)$ are homeomorphic.
 \end{theorem}

 \subsection*{Notation}
 
 We work over the field of complex numbers. All categories and functors are $\mathbb C$-linear   and
 any functor of triangulated categories is exact  as in \cite[Definition 1.39]{H}. 
 Moreover, we denote by $X(\cc)$ the set of closed points of a scheme $X$.

  \subsection*{Acknowledgments}
  
I am grateful to  Kabeer Manali, Amnon Neeman, Laura Pertusi, Taro Sano, Paolo Stellari and Luca Tasin for useful conversations.  
  The author was  partially supported by INdAM-GNSAGA, 
PRIN 2020: ``Curves, Ricci flat varieties and their interactions''
and PRIN 2022: ``Symplectic varieties: their interplay with Fano manifolds and derived categories''.

\section{The Serre functor}

Let $X$ be a smooth projective algebraic   variety over $\mathbb C$.
 We denote by ${\rm D}^b(\mathrm{  Qcoh}(X))$ and $\mathrm{  D}^b(\mathrm{  Coh}(X))$ the bounded derived 
categories of quasi-coherent and coherent sheaves on $X$, respectively.
For a complex $\sF$ in $\mathrm{  D}^b(\mathrm{  Qcoh}(X))$, we denote by $\sH^j(\sF)$ the cohomologies of $\sF$ and 
by
$$\mathrm{  Supp}(\sF) := \bigcup_{j\in \mathbb Z} \mathrm{  Supp}(\sH^j(\sF)) \subset X$$
its support.
If $\sF$ is in $\mathrm{  D}^b(\mathrm{  Coh}(X))$, then  $\mathrm{  Supp}(\sF)$ is an algebraic closed subset and for any closed point $p\in X$:
$$p\in \mathrm{Supp} (\sF) \; \Leftrightarrow \;  \mathrm{  RHom}(\sF, \cc(p)) \neq 0   $$
(see \cite[Definition 2.2 and Lemma 2.4]{LLZ} or \cite[Exercise 3.30]{H}).

If $Z\subset X$ is an algebraic  closed subset,  we denote by 
 $\mathrm{  D}_Z^b({\rm Qcoh}(X))$ and $\mathrm{  D}_Z^b(\mathrm{  Coh}(X))$  the bounded derived categories of  complexes of 
 quasi-coherent and coherent sheaves  on $X$, respectively, with cohomologies supported in $Z$.
 Following \cite[Corollary 3.4]{Ba}, if $\mathrm{ Coh}_Z(X)$ denotes the abelian category of coherent sheaves supported in $Z$, then 
 there is an equivalence of derived categories
 \begin{equation}\label{eq:dercat}
\mathrm{  D}_Z^b(\mathrm{  Coh}(X)) \; \simeq \; \mathrm{  D}^b ( \mathrm{  Coh}_Z(X)).
\end{equation}
Moreover there is a natural inclusion functor 
$$\iota_Z \colon \mathrm{  D}_Z^b(\mathrm{  Qcoh}(X)) \to \mathrm{  D}^b(\mathrm{  Qcoh}(X))$$ 
which admits a right adjoint
$$\iota_Z^! \colon \mathrm{  D}^b(\mathrm{  Qcoh}(X)) \to \mathrm{  D}_Z^b(\mathrm{  Qcoh}(X))$$ 
 (\emph{cf.} \cite[p. 1353]{CS}, \cite[\S 3]{Lip}, \cite{N}).
Hence there are natural isomorphisms
$$\mathrm{  Hom}_{\mathrm{  D}^b(\mathrm{  Qcoh}(X))}( \iota_Z \sF , \sG) \; \simeq \; 
\mathrm{  Hom}_{\mathrm{  D}^b_Z(\mathrm{  Qcoh}(X))}(\sF, \iota_Z^!\sG)$$ 
for all objects   $\sF$   in ${\rm D}^b_Z({\rm Qcoh}(X))$ and 
 $\sG$ in $\mathrm{  D}^b(\mathrm{  Qcoh}(X))$.
The support of a complex $\sF$ in $\mathrm{  D}_Z^b(\mathrm{  Coh}(X))$ is  the algebraic closed subset
$$ \mathrm{  Supp}(\sF)  \; := \;  \mathrm{  Supp}(\iota_Z \sF) \; \subset \; Z .$$
 
%
\begin{prop}\label{prop:formulas}
The following formulas hold for all $\sF, \sG \in \mathrm{  D}^b(\mathrm{  Qcoh}(X))$:
\begin{enumerate}
\item $\iota_Z \iota_Z^! \sF \simeq \sF$ if and only if $\mathrm{  Supp}(\sF) \subset Z$,\\
\item $\iota_Z\iota_Z^! (\sF \stackrel{ \mathrm{  L}}{\otimes} \sG) \simeq \sF \stackrel{\mathrm{ L}}{\otimes} \iota_Z\iota_Z^!\sG$.
\end{enumerate}
\end{prop} 

\begin{proof}
The first assertion is the global version of \cite[Corollary 3.2.1]{Lip}, while
the second one is showed in \cite[Lemma 2.4]{CS}.

\end{proof}

We now discuss the existence of a Serre functor in $\mathrm{  D}_Z^b(\mathrm{  Coh}(X))$ (\emph{cf}. \cite[Definition 1.28]{H}).
Consider the functor 
$$S_Z \colon \mathrm{  D}_Z^b(\mathrm{  Coh}(X)) \to \mathrm{  D}_Z^b(\mathrm{  Coh}(X)), \quad \quad 
S_Z ( \sF) = \iota_Z^! (\iota_Z \sF \otimes \omega_X[\dim X]).$$

\begin{prop}
The functor $S_Z$ is a Serre functor for $\mathrm{  D}_Z^b(\mathrm{  Coh}(X))$.
\end{prop}

\begin{proof}
Let $\sF$ an $\sG$ be arbitrary objects in  $\mathrm{  D}_Z^b(\mathrm{  Coh}(X))$.
We first check that $S_Z$ is an equivalence.
To this end define  $S_Z^{-1} (\sG) = \iota_Z^! (\iota_Z\sG \otimes \omega_X^{-1}[-\dim X])$ 
and note the following isomorphisms   (\emph{cf}. Proposition \ref{prop:formulas} (ii)):
\begin{align*}
S_Z^{-1}(S_Z(\sF)) \; = \;  \iota_Z^!(\iota_Z \iota_Z^!(\iota_Z \sF \otimes \omega_X[\dim X])\otimes \omega_X^{-1}[-\dim X])   \simeq  \\
\iota_Z^! \iota_Z \iota_Z^! (\iota_Z \sF \otimes \omega_X[ \dim X]\otimes \omega_X^{-1}[- \dim X])\simeq \\
\iota_Z^! \iota_Z \sF \; \simeq \;  \sF.
\end{align*}
Similarly, one checks that $S_Z ( S_Z^{-1} (\sG) )\simeq \sG$.  
  Finally, the following isomorphisms hold   thanks to Serre duality on $X$:
\begin{align*}
\mathrm{ Hom}_{\mathrm{  D}_Z^b(\mathrm{  Coh}(X))} (\sG, S_Z(\sF) ) \simeq 
\mathrm{ Hom}_{\mathrm{  D}^b(\mathrm{  Coh}(X))} (\iota_Z \sG ,\iota_Z\iota_Z^! ( \iota_Z \sF \otimes \omega_X[\dim X]) )   \simeq \\
\mathrm{ Hom}_{\mathrm{  D}^b(\mathrm{  Coh}(X))} (\iota_Z \sG ,( \iota_Z\iota_Z^! \iota_Z \sF \otimes \omega_X[\dim X]) )  \simeq   \\
\mathrm{ Hom}_{\mathrm{  D}^b(\mathrm{  Coh}(X))} (\iota_Z \sG ,  \iota_Z \sF \otimes \omega_X[\dim X]  ) \simeq \\
 \mathrm{ Hom}_{\mathrm{  D}^b(\mathrm{  Coh}(X))} (\iota_Z\sF , \iota_Z \sG )^{*} \simeq \\
  \mathrm{ Hom}_{\mathrm{  D}_Z^b(\mathrm{  Coh}(X))}(\sF , \sG)^*.
\end{align*}

\end{proof}

Let $Y$ be another smooth projective algebraic variety and let $W\subset Y$ be an algebraic  closed subset. 

\begin{cor}\label{cor:eqdim}
If  $F\colon \mathrm{  D}_Z^b(\mathrm{  Coh}(X)) \to \mathrm{  D}_W^b(\mathrm{  Coh}(Y))$ is an equivalence, then $\dim X = \dim Y$.
\end{cor}
\begin{proof}
Note the isomorphism  $S_W \circ F \simeq F \circ S_Z$ (\cite[Lemma 1.30]{H}) where 
$S_W$ is the Serre functor of $\mathrm{  D}_W^b(\mathrm{  Coh}(Y))$. Let $p\in Z$ be a closed point  and note the isomorphisms
\begin{align*}
\iota_W^! (\iota_W F(\iota_Z^!\cc(p)) \otimes \omega_Y[\dim Y]) \simeq \\
S_W(F(\iota_Z^!\cc(p)) )
 \simeq \\ 
 F( S_Z(\iota_Z^! \cc(p) ) ) \simeq \\
 F(\iota_Z^!\cc(p)[\dim X] )  \simeq \\
F(\iota_Z^!\cc(p))[\dim X]. 
\end{align*}
By applying $\iota_W$ we find 
\begin{gather*}
\iota_W \iota_W^! (\iota_W F(\iota_Z^!\cc(p)) \otimes \omega_Y[\dim Y])  \simeq \iota_W F(\iota_Z^!\cc(p))[\dim X].
\end{gather*}
Moreover, by Proposition \ref{prop:formulas} (ii), the left-hand side is isomorphic to $\iota_W F(\iota_Z^!\cc(p)) \otimes \omega_Y[\dim Y]$. 
We conclude  that 
$$ \iota_W F(\iota_Z^!\cc(p)) \otimes \omega_Y \simeq   \iota_W F(\iota_Z^!\cc(p)) [\dim Y-\dim X].$$
This immediately implies $\dim X = \dim Y$ as in \cite[Proposition 4.1]{H} since $\iota_W F(\iota_Z^!\cc(p))$ is a bounded complex.

\end{proof}

\section{Point like objects in the supported case}
Let $Z \subset X$  be as in the previous section.

\begin{defn}\label{def:pl}
An object $P$ of $\mathrm{  D}_Z^b(\mathrm{  Coh}(X))$ is  \emph{point like} if:
\begin{enumerate}
\item $S_Z(P)\simeq P[\dim X]$, \\
\item  $\mathrm{  Hom}_{\mathrm{  D}_Z^b(\mathrm{  Coh}(X))} (P, P[k])  =  0$ for $k<0$, and \\
\item ${\rm Hom}_{\mathrm{  D}_Z^b(\mathrm{  Coh}(X))}(P,P)   =  \mathbb C$.
\end{enumerate}
 An object satisfying the last condition is called \emph{simple}. 
 \end{defn}
 
 \begin{prop}\label{prop:preservesplo}
 Any equivalence $F\colon \mathrm{  D}_Z^b(\mathrm{  Coh}(X)) \to \mathrm{  D}_W^b(\mathrm{  Coh}(Y))$ sends a point like object to a point like object.
 \end{prop}
 
 \begin{proof}
 The conditions $(ii)$ and $(iii)$ are easy to check. For $(i)$, let $P$ be a point like object in $\mathrm{  D}_Z^b(\mathrm{  Coh}(X))$.
 As the Serre functor commutes with any equivalence (\cite[Lemma 1.30]{H}), we have by Corollary \ref{cor:eqdim}:
 $$   S_W ( F (P ) ) \simeq F (S_Z(P ) ) \simeq F ( P[ \dim X]) \simeq  F(P )[\dim Y].$$
  \end{proof}
 
\begin{prop}\label{prop:pointlike}
 The objects isomorphic to  $\iota_Z^! \cc(p) [m]$ with $p\in Z$  closed point  and 
$m\in \mathbb Z$ are point like objects in $\mathrm{  D}_Z^b(\mathrm{  Coh}(X))$.
\end{prop}

\begin{proof}
We prove the proposition for $m=0$, the other cases being similar. 
Let $p\in Z$ be a closed point. 
By Proposition \ref{prop:formulas}  $\iota_Z^! \cc(p)$ satisfies the first condition of Definition \ref{def:pl}: 
\begin{align*}
S_Z (\iota_Z^! \cc(p)) \simeq  \iota_Z^! (\iota_Z \iota_Z^! \cc(p) \otimes \omega_X[ \dim X]) \simeq \\
\iota_Z^! \iota_Z \iota_Z^! (\cc(p) \otimes \omega_X[\dim X]) \simeq \\
\iota_Z^! \cc(p) [\dim X].
\end{align*} 
For the second and third  conditions we see that   by adjunction
\begin{align*}
\mathrm{ Hom} (\iota_Z^! \cc(p) , \iota_Z^! \cc(p) [k] ) \simeq 
\mathrm{  Hom} (\iota_Z \iota_Z^! \cc(p) , \cc(p) [k] )\simeq 
\mathrm{  Hom} ( \cc(p) , \cc(p) [k] )
\end{align*} 
for all $k\in \mathbb Z$.
These spaces  are trivial for $k<0$ and isomorphic  to $\mathbb C$ for $k=0$.
\end{proof}

We recall the following lemma \cite[Lemma 4.5]{H} which will be helpful in the rest of the paper.

\begin{lemma}\label{lem:pointlike}
Let $P$ be an object in $\mathrm{  D}^b( \mathrm{  Coh}(X) )$ such that
 $\mathrm{  Hom}(P,P) = \mathbb C$ and $\mathrm{  Hom}(P,P [ k ] ) = 0$ for $k<0$. 
Moreover  suppose   $\mathrm{  Supp}(P)$ is zero-dimensional. 
Then $P \simeq \cc(p)[m]$ for some  closed point $p \in \mathrm{  Supp}(P)$  and $m\in \mathbb Z$.

\end{lemma}

We begin the classification of point like objects in $\mathrm{  D}_Z^b(\mathrm{  Coh}(X))$ when $\dim Z=0$.

\begin{prop}\label{prop:isosupport0}
If  $\dim Z=0$,
then the point like objects in $\mathrm{  D}_Z^b(\mathrm{  Coh}(X))$  are the objects isomorphic
to $\iota^!_Z \cc(p)[m]$ with  $p\in Z$ closed point   and $m\in \mathbb Z$. 
\end{prop}
\begin{proof}
Let $P$ be a point like object in $ \mathrm{  D}_Z^b(\mathrm{  Coh}(X)) $. Then $\iota_{Z}P$ has zero-dimensional support and 
$\mathrm{  Hom}_{\mathrm{  D}^b(\mathrm{  Coh}(X)) } (\iota_Z P , \iota_Z P[ k ] ) \simeq 
 \mathrm{  Hom}_{\mathrm{  D}_Z^b(\mathrm{  Coh}(X))} (P , P[ k ] )$ is trivial for $k<0$,  
 and isomorphic to $\mathbb C$ for $k=0$.
By Lemma \ref{lem:pointlike}  there is an isomorphism $\iota_{Z}P \simeq \cc(p)[m]$ for some closed point $p\in Z$  and  
$m\in \mathbb Z$. It follows  $P\simeq \iota_Z^!\cc(p)[m]$.
\end{proof}

	\begin{prop}\label{prop:closedpoints}
If $F\colon \mathrm{  D}_Z^b(\mathrm{  Coh}(X)) \to  \mathrm{  D}_W^b(\mathrm{  Coh}(Y))$ is an equivalence and  $\dim Z=0$, then $\dim W=0$ and 
there is a bijection between $Z(\cc)$ and $ W(\cc)$.

\end{prop}

\begin{proof}
Let $q\in W$ be a  closed point and consider the complex $E_q := F^{-1}(\iota_W^!\cc(q))$ where $F^{-1}$ is a quasi-inverse to $F$. 
Then, by Proposition \ref{prop:preservesplo},  $E_q$ is a point like object in ${\rm D}^b_Z ({\rm Coh}(X))$ and 
by Proposition \ref{prop:isosupport0}  $E_q\simeq \iota_{Z}^!\cc(p)[m]$ for a unique closed point $p\in Z$ and
$m\in  \mathbb Z$.  In this way we define a function $f\colon W(\cc) \to Z(\cc)$.
We first check that $f$ is injective.  If there exists another    closed point $q' \in W$  such that 
$f(q)=f(q')$, then as before we would have  $F( \iota_Z^!\cc(p)[m'] ) =  \iota_W^!\cc(q')$
 for some $m'\in \mathbb Z$.  Hence we have
 $$\{q \} = \mathrm{  Supp } (\iota_W^!\cc(q)[-m] )   =  \mathrm{  Supp } ( F( \iota_{Z}^!\cc(p) ) ) =
  \mathrm{  Supp} ( \iota_W^!\cc(q')[-m'] )  = \{q'\}.$$ 
It follows that $W(\cc)$ is a finite  and
$\dim W=0$. Moreover, by repeating the argument for  the quasi-inverse $F^{-1}$ one checks also the surjectivity of $f$.

\end{proof}

\begin{ex}
Let $X$ be an abelian variety of positive dimension and let  $Y=\widehat{X}$ be the dual abelian variety. 
Moreover let $\sP$ be a normalized Poincar\'e line bundle.
Consider the Fourier--Mukai--Poincar\'{e} transform 
$\mathrm{  R}\widehat {\mathcal S}\colon \mathrm{  D}^b(\mathrm{  Coh}(X)) \to \mathrm{  D}^b(\mathrm{  Coh}(Y))$ defined as 
$\mathrm{  R}\widehat {\mathcal S} (\sF) = \mathrm{  R}q_* ( p^* \sF \otimes \sP)$ for any $\sF$ in $\mathrm{  D}^b(\mathrm{  Coh}(X))$, 
where $p$ and $q$ are the projections onto $X$ and
 $Y$, respectively. Then as in \cite[Example 2.9]{M}
$\mathrm{  R}\widehat {\mathcal S}$ is an equivalence  and $\mathrm{  R}^{\dim X}\widehat {\mathcal S}$ gives an equivalence of
 abelian categories between the category  $\mathrm{  Unip}(X)$ of unipotent bundles of finite rank on $X$,  
and $\mathrm{  Coh}_{\{ \hat{0} \}}(Y)$ (where 
$\hat{0}$ is the origin of $Y$).
Then  $\mathrm{  D}^b(\mathrm{  Unip}(X)) \simeq \mathrm{  D}^b ( \mathrm{  Coh}_{\{ \hat{0} \}}(Y) )$ and, by 
Corollary \ref{cor:eqdim} and Proposition \ref{prop:closedpoints}, there
is no smooth projective variety $M$ such that 
$\mathrm{  D}^b(\mathrm{  Unip}(X))$ is equivalent to $\mathrm{  D}^b (\mathrm{  Coh}(M))$.

\end{ex}

In the positive-dimensional case the  characterization
of point like objects in $\mathrm{  D}_Z^b(\mathrm{  Coh}(X))$ is possible under a positivity assumption on the restriction of the canonical bundle to the support.

\begin{prop}\label{prop:pointlikeample}
Suppose   $\omega_X|_Z$ or $\omega_X^{-1}|_Z $ is ample. 
Then the point like objects  of $\mathrm{  D}_Z^b(\mathrm{  Coh}(X))$ are the objects isomorphic
to $\iota_Z^! \cc(p) [r]$ where $p\in Z$ is a closed point  and $r\in \mathbb Z$.
\end{prop}

\begin{proof}
Let $P$ be a point like object  in $\mathrm{  D}_Z^b(\mathrm{  Coh}(X))$
and denote by $\sH^j$ the cohomology sheaves of $\iota_Z P$. Note that the
$\sH^j$'s are coherent sheaves on $X$ supported in $Z$. Since $S_Z(P) \simeq P[ \dim X ]$, 
we have the following isomorphisms:
\begin{align*}
\iota_Z^!(\iota_Z P\otimes \omega_X ) \simeq P \\ 
\iota_Z \iota_Z^! (\iota_Z P\otimes \omega_X) \simeq \iota_Z P\\
 \iota_Z \iota_Z^!  \iota_Z P\otimes \omega_X  \simeq \iota_Z P\\
\iota_Z P\otimes \omega_X  \simeq \iota_Z P.
\end{align*}
 By taking cohomology we have 
\begin{align}\label{eq:tensorh}
\sH^j \otimes \omega_X  \simeq \sH^j \quad \mbox{ for all } j.
\end{align}
 Now we show that $\sH^j$ is supported in dimension zero for any $j$.
Suppose $\omega_X|_Z$ is ample, the other case being similar. Let $k>0$ be an integer such that  
$N:= \omega_X^{\otimes k}|_Z$ is very ample  and let 
$i\colon Z\hookrightarrow X$ be the inclusion map (here $Z$ is equipped with the reduced induced subscheme structure). 
Then the Hilbert polynomial of $P_{i^*\sH^j}(m)=\chi ( i^*\sH^j \otimes N^{\otimes m})$ has degree 
equal to 
$$s_j \; := \; \dim \mathrm{  Supp}(i^*\sH^j) = \dim \mathrm{  Supp}(\sH^j)$$ 
(\cite[p. 276, Proposition 6]{S}). 
Moreover, by tensoring \eqref{eq:tensorh} with positive  powers of $\omega_X$ 
and by restricting the isomorphisms  to $Z$, we find 
$$i^*\sH^j\otimes N \; \simeq \; i^*\sH^j\quad \mbox{ for all }j.$$
Therefore $P_{i^*\sH^j}(m)=P_{i^*\sH^j\otimes N}(m) = P_{i^*\sH^j} (m+1)$ for all $m \in \mathbb Z$  which  is impossible if $\deg P_{i^*\sH^j}>0$. 
Hence $s_j=0$ for all $j$ and  $\iota_Z P \simeq \cc(p)[r]$ for some  closed point $p\in Z$ and $r\in \mathbb Z$  by  Lemma \ref{lem:pointlike}. 
It follows that  $P\simeq \iota_Z^! \cc(p) [r]$.

\end{proof}

\begin{ex}\label{ex:R}
We construct further instances of equivalences between derived categories with support extending the equivalences \eqref{eq:toda}.
Denote by $$R \colon \mathrm{  Aut}^0(X) \times \mathrm{  Pic}^0(X) \to \mathrm{  Aut}^0(Y) \times \mathrm{  Pic}^0(Y)$$
  the Rouquier isomorphism induced by an equivalence
$F\colon \mathrm{  D}^b( \mathrm{   Coh}(X)) \to \mathrm{  D}^b(\mathrm{   Coh}(Y))$ (\emph{cf.} \cite[Th\'{e}orème 4.18]{R2} or \cite[Proposition 9.45]{H}, 
and \cite[p.531, footnote (1)]{PS}). By following \cite[Proposition 3.1]{L}, 
if $\alpha \in \Pic^0(X)$ is a topologically trivial  line bundle such that $H^0(X,\omega_X^{\otimes k_0}\otimes \alpha) \neq 0$ for some 
$k_0 \in \mathbb Z$, then $R(  {\rm id}_X , \alpha) =
( \mathrm{  id}_Y , \beta )$  for some $\beta\in \Pic^0(Y)$ and  moreover there are  isomorphisms
$R_k\colon H^0(X , \omega_X^{\otimes k} \otimes \alpha) \to H^0(Y, \omega_Y^{\otimes k}\otimes \beta)$ for all $k\in \mathbb Z$.
As in \cite{T}, one can prove that 
 if $E\in \big| \omega_X^{\otimes k} \otimes \alpha \big|$ and $E'=R_k(E) \in \big| \omega_Y^{\otimes k} \otimes \beta \big|$ 
is the corresponding divisor, then 
$F$ restricts to an equivalence of triangulated categories $\mathrm{  D}_{ \mathrm{  Supp}(E) }^b(\mathrm{   Coh}(X))\simeq 
\mathrm{  D}_{\mathrm{  Supp}(E')}^b(\mathrm{   Coh}(Y))$.

\end{ex}

\section{Gabriel's theorem with support}

We present a version of Gabriel's reconstruction theorem \cite{G}  with support.
Other variants of  Gabriel's result appear in \cite{Per} and \cite{CP}.

\begin{theorem}\label{thm:recabelian}
Let $X$ and $Y$ be  smooth projective varieties and let $Z\subset X$ and $W\subset Y$ be algebraic closed subsets.
If  there is an equivalence  of abelian categories  $\varphi \colon \mathrm{  Coh}_Z(X) \to \mathcal \mathrm{  Coh}_W(Y) $, then 
$Z(\cc)$ and $W(\cc)$ are homeomorphic.

\end{theorem}

\begin{proof}

A sheaf $\sF\in \mathrm{  Coh}_Z(X) $ (\emph{resp.} in $\mathrm{  Coh}(X) $) 
is said \emph{indecomposable} if any non-trivial surjection $\sF \twoheadrightarrow \sG$ with $\sG$ in $\mathrm{  Coh}_Z(X) $ (\emph{resp.}  
in $\mathrm{  Coh}(X) $)
is an isomorphism. Recall that the indecomposable sheaves in $\mathrm{  Coh}(X) $ are the skyscraper  sheaves  of closed points $\cc(x)$. 
It follows that the indecomposable sheaves in $\mathrm{  Coh}_Z(X) $ are the sheaves $\iota_Z^! \cc(x)$ where  $x\in Z$ is a closed point.
Indeed, let $\sF \in \mathrm{  Coh}_Z(X)$ be an
indecomposable sheaf. It is enough to show that  
$\iota_Z \sF$ is indecomposable in $\mathrm{  Coh}(X) $. To this end, consider a short exact sequence of type
$$ 0 \to \sK \to \iota_Z \sF \to \sG \to 0$$ 
with $\sG$ non trivial    in $\mathrm{  Coh}(X) $. Under the full embedding $\mathrm{  Coh}(X) \hookrightarrow \mathrm{  D}^b ( \mathrm{  Coh}(X) )$, 
the above sequence
becomes a distinguished triangle $ \sK \to \iota_Z \sF \to \sG \to \sK[1]$. In turn the functor $\iota_Z^!$ induces 
  a distinguished triangle in $\mathrm{  D}^b_Z( \mathrm{  Coh}(X))$
$$\iota_Z^! \sK \to \sF \to \iota_Z^! G\to \iota_Z^!K[1].$$ 
Since both $\sK$ and $\sG$ have support in $Z$, we check that $\iota_Z^! \sK$ and $\iota_Z^!\sG$ 
are sheaves in $\mathrm{  Coh}_Z(X)$. 
Under the full embedding $\mathrm{  Coh}_Z(X) \hookrightarrow \mathrm{  D}^b( \mathrm{  Coh}_Z(X)) \simeq  \mathrm{  D}^b_Z( \mathrm{  Coh}(X))$, 
this distinguished triangle becomes 
 a short exact sequence $0 \to \iota_Z^! \sK \to \sF \to \iota_Z^! \sG \to 0$ with $\iota_Z^!\sG \neq 0$. 
Hence $\iota_Z^!\sK \simeq 0$ and consequently $\sK$ is trivial.

The equivalence $\varphi$ sends indecomposable sheaves to indecomposable sheaves. Hence we can define a function $f\colon Z(\cc)\to W(\cc)$ on closed points 
such that  $f(q)=p$ if and only if $\varphi(\iota_Z^!\cc(q)) = \iota_W^! \cc(p)$. Note that $f$ is bijective. Now we prove that 
$f$ is bicontinuous by following an idea of \cite{LLZ}.  

Denote by $\varphi^{-1}$ a quasi-inverse of $\varphi$ and 
let $W_0\subset W(\cc)$ be a closed subset. Hence $W_0=W'\cap W(\cc)$ for some   algebraic closed subset $W'  \subset W$. 
Since for every $q\in Z(\cc)$  
\begin{align*}
\mathrm{  Hom}( \varphi^{-1} (\iota_W^! \sO_{ W'} ) ,\iota_Z^! \cc(q)) \simeq \\
\mathrm{  Hom}( \iota_W^! \sO_{W'} , \varphi(\iota_Z^!\cc(q) ) ) \simeq \mathrm{  Hom}( \iota_W^! \sO_{ W'} , \iota_W^!\cc(f(q))  ) \simeq \\
\mathrm{  Hom}(  \sO_{ W'} , \cc(f(q))  ),
\end{align*}
we have
$$q \in {\rm Supp} ( \varphi^{-1} (\iota_W^! \sO_{ W'} )) \cap Z(\cc) \;  \Leftrightarrow \; f(q) \in  W' \cap W(\cc) = W_0.$$ 
It follows that $f^{-1}(W_0) = \mathrm{  Supp} (  \varphi^{-1} (\iota_W^! \sO_{ W'} ) ) \cap Z(\cc) $ is  closed  in $Z(\cc)$
and $f$ is continuous. 

Similarly, also $f^{-1}$ is continuous. This goes as follows. 
Let $Z_0=Z'\cap  Z(\cc)$ be   a closed subset in $Z(\cc)$ where $Z'$ is closed in $Z$. Note
$$\mathrm{  Hom}(  \sO_{ Z' } ,  \cc(q) )  \simeq \mathrm{  Hom}( \iota_Z^! \sO_{ Z'} , \iota_Z^! \cc(q) ) \simeq 
\mathrm{  Hom} ( \varphi( \iota_Z^! \sO_{ Z'} ) , \iota_W^! \cc(f(q)) ) \quad \mbox{ for every } \quad q \in Z(\cc).$$ 
It follows  
$$f(q) \in \mathrm{  Supp} ( \varphi (\iota_Z^! \sO_{Z'} )) \cap W(\cc)  \; \Leftrightarrow \; q \in  Z'  \cap Z(\cc) =Z_0.$$ 
Hence
$f(Z_0) = \mathrm{  Supp} ( \varphi ( \iota_Z^! \sO_{ Z' }) ) \cap W(\cc) $ is  closed   in $W(\cc)$.

\end{proof}

\section{ Reconstruction of the support, I}

We aim to prove Theorem \ref{thm:intr1}.
Denote by
$F \colon \mathrm{  D}_Z^b(\mathrm{  Coh}(X)) \to \mathrm{  D}_W^b(\mathrm{  Coh}(Y))$  an equivalence
and let $p\in W(\cc)$. Then there exists an object $E_p$ in $\mathrm{  D}_Z^b(\mathrm{  Coh}(X))$ such that 
$F(E_p) = \iota_W^! \cc(p)$. By Proposition \ref{prop:preservesplo}, 
 $E_p$ is a point like object and by Proposition \ref{prop:pointlikeample} $E_p \simeq \iota_Z^! \cc(q)[m_p]$ for some 
$q\in Z(\cc)$  and $m_p\in \mathbb Z$.
 Define $U(\cc)\subset Z(\cc)$ as the set of closed  points $q$ such that $F(\iota_Z^! \cc(q)) \simeq \iota_W^!\cc(p)[s_p]$ for some 
    $p\in W(\cc)$ and $s_p \in \mathbb Z$.
In this way 
we can define a surjective function $f\colon U(\cc) \to W(\cc)$ such that $f(q)=p$ if and only if $F(\iota_Z^!\cc(q))=\iota_W^!\cc(p)[s_p]$.

\begin{prop}
The function $f$ is injective.
\end{prop}

\begin{proof}
Suppose there exist $q,q'\in U(\cc)$ such that $f(q)=f(q')$. Then there exist $m_1,m_2\in \mathbb Z$ such that 
$F(\iota_Z^! \cc(q)[m_1]) = F ( \iota_Z^! \cc(q')[m_2] )$ and $\iota_Z^! \cc(q)[m_1]  \simeq  \iota_Z^! \cc(q')[m_2]$. By looking at the supports it results
$q=q'$.

\end{proof}

\begin{prop}
The set $U(\cc)$ equals  $Z(\cc)$.
\end{prop}

\begin{proof}
 
Assume by contradiction there exists $q' \in Z(\cc) \backslash U(\cc)$. If  $p\in W(\cc)$  is any point, then $f^{-1}(p)\in U$ and 
\begin{align*}
\mathrm{  RHom}( F (\iota_Z^! \cc(q') ) , \iota_W^! \cc(p) ) \simeq \\
\mathrm{  RHom}( F (\iota_Z^! \cc(q') ) , F  (\iota_Z^! \cc(f^{-1}(p)) [-s_p]))\simeq  \\ 
\mathrm{  RHom}(  \cc(q')  , \cc(f^{-1}(p))[-s_p] ) \; = \;  0.
\end{align*}
It follows that $\mathrm{  Supp} ( F( \iota_Z^!  \cc(q') ) )   = \emptyset$ and $\iota_Z^! \cc(q') \simeq 0$.

\end{proof}

\begin{prop}
The   functions $f$ and $f^{-1}$ are continuous.
\end{prop}

\begin{proof}

Let $W_0\subset W$ be a closed subset.
Since for every $q\in Z(\cc)$ 
\begin{align*}
\mathrm{  RHom}( F^{-1} (\iota_W^! \sO_{ W_0} ) ,\iota_Z^! \cc(q)) \; \simeq  \; 
\mathrm{  RHom}( \iota_W^! \sO_{W_0} , F(\iota_Z^!\cc(q) ) ) \; \simeq \;
\mathrm{  RHom}( \iota_W^! \sO_{ W_0} , \iota_W^!\cc(f(q))  ),
\end{align*}
we have
$$q \in \mathrm{  Supp} ( F^{-1} (\iota_W^! \sO_{ W_0} ) )  \cap Z(\cc) \; \Leftrightarrow \;  f(q) \in  W_0 \cap W(\cc).$$ 
It follows that $f^{-1}(W_0 \cap W(\cc) )= \mathrm{  Supp} (  F^{-1} (\iota_W^! \sO_{ W_0 } )  )  \cap Z(\cc)  $ is closed in $Z(\cc)$
and $f$ is continuous. 

Now we show that $f^{-1}$ is continuous. 
Let $Z_0 \subset Z$ be a closed subset and note the isomorphism
$$\mathrm{  RHom}( \iota_Z^! \sO_{ Z_0 } , \iota_Z^! \cc(q) ) \simeq \mathrm{  RHom} (F( \iota_Z^! \sO_{ Z_0 } ) , \iota_W^! \cc(f(q)) )
\quad \mbox{ for every } \quad q\in Z(\cc).$$ 
Then it follows that 
$$f(q) \in \mathrm{  Supp} ( F (\iota_Z^! \sO_{Z_0} )  ) \cap W(\cc) \;  \Leftrightarrow  \; q \in  Z_0 \cap Z(\cc) $$ 
and 
$f(Z_0 \cap Z(\cc) ) = \mathrm{  Supp} ( F( \iota_Z^! \sO_{ Z_0 }  ) )  \cap W(\cc)$  is closed in $W(\cc)$.

\end{proof}

\section{ Reconstruction of the support, II}

We aim to prove Theorem \ref{thm:intr2}. The proof is divided into four steps and it follows the general strategy of \cite[\S3]{LLZ}.
 
\subsection{Step 1}
We need the following characterization of skyscraper sheaves.

\begin{prop}\label{prop:charstr}
Suppose $X$ is a smooth projective variety and  let $\sO_X(1)$  be a   very ample line bundle.
Suppose $Z \subset X$ is an algebraic closed subset and 
let $E$ be a simple object in $\mathrm{  D}^b_Z ( {\rm  Coh}(X))$ such that 
$\mathrm{  Hom}_{{\rm D}^b_Z ( \mathrm{   Coh}(X))}(E , E[j] ) = 0$ for $j<0$.
If there exist non-negative integers $a_0 , \ldots , a_{\dim X}$ such that 
$$\mathrm{  RHom}_{\mathrm{  D}^b_Z ( \mathrm{   Coh}(X))}(\iota_Z^!  ( \sO_Z (-m) )  ,  E) \; \simeq \; 
\bigoplus_{k=0}^{\dim X} \mathbb C^{\oplus a_k}[-k] $$ for $m\gg 0$, 
then $E \simeq \iota_Z^! \cc(x)[r]$  for some  closed point  $x\in Z$ and $r\in \mathbb Z$.
\end{prop}

\begin{proof}
Denote by $i\colon Z \hookrightarrow X$ the inclusion map and note the isomorphisms:
\begin{align*}
\mathrm{  Hom}_{{\rm D}^b_Z ( \mathrm{   Coh}(X))} (\iota_Z^!  i_* \sO_Z (-m)   ,  E) \simeq \\
\mathrm{  Hom}_{{\rm D}^b ( \mathrm{   Coh}(X))}( i_* \sO_Z (-m) ,  \iota_Z E  )\simeq \\
\mathrm{  Hom}_{{\rm D}^b ( \mathrm{   Coh}(X))}( \sO_X ,  ( i_*\sO_Z)^{\stackrel{ {\rm R}}{\vee}} \stackrel{\mathrm{  L}}{\otimes }  \iota_Z E \otimes \sO_X(m) ) 
\end{align*}
Let $\sH^q$ be the $q$-th cohomology of  the complex
$( i_*\sO_Z)^{\stackrel{ \mathrm{  R}}{\vee}} \stackrel{\mathrm{  L}}{\otimes }  \iota_Z E $
and consider the spectral sequence
\begin{align*}
E_2^{p,q}  \; = \;  \mathrm{  Hom}_{\mathrm{  D}^b ( \mathrm{   Coh}(X))} ( \sO_X ,\sH^q \otimes \sO_X(m)[p] ) \; \Rightarrow  \; \\
\mathrm{  Hom}_{\mathrm{  D}^b ( \mathrm{   Coh}(X))} ( \sO_X , ( i_*\sO_Z)^{\stackrel{ {\mathrm{  R}}}{\vee}} \stackrel{\mathrm{  L}}{\otimes } 
 \iota_Z E \otimes \sO_X(m)[p+q] ) 
\end{align*}
For $m\gg 0$ the elements $E_2^{p,q}$ are trivial for all $p \neq 0$ and all $q$. Hence the spectral sequence degenerates at the second page 
and 
$$
H^0( X , \sH^q \otimes \sO_X(m) ) \simeq 
\mathrm{  Hom}_{\mathrm{  D}^b ( \mathrm{   Coh}(X))} ( \sO_X , ( i_*\sO_Z)^{\stackrel{ \mathrm{  R}}{\vee}} \stackrel{\mathrm{  L}}{\otimes }  \iota_Z E \otimes \sO_X(m)[q] ) 
$$
for all $q$. By hypotheses, the right hand side of the above isomorphisms vanish for $q < 0$ or $q >\dim X$, while in middle degrees we have
$$H^0( X , \sH^q \otimes \sO_X(m) )  \simeq \mathrm{  Hom}_{ \mathrm{  D}^b_Z ( \mathrm{   Coh}(X)) } (\iota_Z^!  ( \sO_Z (-m) )  ,  E[q]) \simeq  
\mathbb C^{\oplus a_q}$$ 
for all $m \gg 0$.
Moreover we can suppose $\sH^q \otimes \sO_X(m)$ is globally generated for $m\gg 0$ and any $q$. Therefore for every $q$ the sheaf $\sH^q$ is either trivial 
or supported on a finite set of points. In fact, if $\sH^q\neq 0$ and $s=\dim \mathrm{  Supp} (\sH^q)$, then by \cite[p. 276, Proposition 6]{S}
the Hilbert polynomial $P_{\sH^q}(m) := \chi (\sH^q \otimes \sO_X(m))$ has degree equal to $s$ and 
$$P_{\sH^q}(m)  = h^0 ( X , \sH^q \otimes \sO_X(m))=a_q \quad \mbox{ for all } \quad m\gg 0.$$ 
This forces 	$s=0$ and 
\begin{equation}\label{eq:suppdim}
\dim \big(  \mathrm{  Supp} ( ( i_*\sO_Z)^{\stackrel{ \mathrm{  R}}{\vee}} \stackrel{\mathrm{  L}}{\otimes }  \iota_Z E ) \big) \leq 0.
\end{equation} 
On the other hand, there are equalities of sets 
\begin{equation}\label{eq:supp1} 
 \mathrm{  Supp} ( ( i_*\sO_Z)^{\stackrel{ \mathrm{  R}}{\vee}} \stackrel{\mathrm{  L}}{\otimes }  \iota_Z E ) \;   = \;
 \mathrm{  Supp} ( ( i_*\sO_Z)^{\stackrel{ \mathrm{  R}}{\vee}} ) \cap \mathrm{  supp}( \iota_Z E)
\end{equation} 
  and 
\begin{equation}\label{eq:suppdual} 
 \mathrm{  Supp} ( ( i_*\sO_Z)^{\stackrel{ \mathrm{  R}}{\vee}} )\; = \; \mathrm{  Supp} (i_*\sO_Z)=Z.
 \end{equation}
 In order to check \eqref{eq:suppdual},  for a closed point $x\in X$ it results 
 $$\mathrm{  RHom} (( i_*\sO_Z)^{\stackrel{ \mathrm{  R}}{\vee}} , \cc(x)) \; \simeq  \; \mathrm{  RHom}( \sO_X , i_*\sO_Z \stackrel{\mathrm{  L}}{\otimes}
   \cc(x)),$$ and 
 therefore
 $$x\in \mathrm{  Supp}(( i_*\sO_Z)^{\stackrel{ \mathrm{  R}}{\vee}} ) \cap X(\cc) \;  \Leftrightarrow \;  x\in Z(\cc)$$ 
 (here we have employed the degeneration of the spectral sequence 
 $$E_2^{'p',q'}= \mathrm{  Hom}(\sO_X , \sH^{q'} (i_*\sO_Z \stackrel{\mathrm{  L}}{\otimes } \cc(x))[p'])
 \Rightarrow \mathrm{  Hom}(\sO_X , i_*\sO_Z \stackrel{\mathrm{  L}}{\otimes } \cc(x)[p'+q']) ).$$
By putting together \eqref{eq:suppdim}, \eqref{eq:supp1} and \eqref{eq:suppdual}, 
we conclude  that  $\dim \mathrm{  Supp}(\iota_Z E)  \leq 0$. 
The conclusion follows by Lemma \ref{lem:pointlike}.
 
\end{proof}

\subsection{Step 2}
By hypotheses we may suppose $F(\iota_Z^! \cc(x)) = \iota_W^! \cc(y)$ for some  closed points  $x\in Z(\cc)$ and $y\in W (\cc)$.

There are isomorphisms
\begin{align*}
\mathrm{  RHom}_{ \mathrm{  D}_W^b(\mathrm{  Coh}(Y))  } (F(\iota_Z^!  \sO_Z ) ,   \iota_W^!  \cc(y)) \simeq 
\mathrm{  RHom}_{\mathrm{  D}_W^b(\mathrm{  Coh}(Y))  } ( F ( \iota_Z^!  \sO_Z ) ,  F( \iota_Z^!  \cc(x)) )  \simeq \\
\mathrm{  RHom}_{\mathrm{  D}_Z^b( \mathrm{  Coh}(X))  } (\iota_Z^!  \sO_Z  ,  \iota_Z^!  \cc(x) )  \simeq   
\mathrm{  RHom}_{\mathrm{  D}^b(\mathrm{  Coh}(X))  } ( \sO_Z  ,   \cc(x) )   \simeq 
\bigoplus_{k=0}^{\dim X} \mathbb C^{\oplus a_k}[-k] 
\end{align*}
for some integers $a_0, \ldots , a_{\dim X}\geq 0$.

By semi-continuity, there exists an algebraic open dense subset $W' \subset W$ containing $y$ such that for every closed point $p\in W'$:
$$\mathrm{  RHom} ( F ( \iota_Z^!  \sO_Z ) ,   \iota_W^!  \cc(p) ) \simeq \bigoplus_{k=0}^{\dim X} \mathbb C^{\oplus b^p_k}[-k]$$ for some integers 
$b^p_0 , \ldots , b_{\dim X}^p$ such that $0\leq b_k^p \leq a_k$ for any $k=0, \ldots , \dim X$.

Fix an embedding $X\hookrightarrow \mathbb P^N$ and 
let 
$$P_0\subset P_1 \subset P_2 \subset \ldots \subset P_{\dim Z}= \mathbb P^N$$ be a flag of subspaces such that 
$P_j \simeq \mathbb P^{N -\dim Z + j}$ and $Z_j := Z \cap P_j$ is of dimension $j$ for every $0\leq j\leq \dim Z$. 
Moreover, the $Z_j$'s are  irreducible  except 
possibly  $Z_0$ (\cite[Theoreme I.6.10]{J}).
 We can also  assume $x\notin Z_j$ for all $j<\dim Z$. 
Then for any $j<\dim Z$ we have
\begin{align*}
\mathrm{  RHom}_{ \mathrm{  D}_W^b (\mathrm{  Coh}(Y)) } ( F (\iota_Z^!  \sO_{Z_j} ) , \iota_W^! \cc(y)  )  \simeq \\
\mathrm{  RHom}_{ \mathrm{  D}_W^b (\mathrm{  Coh}(Y)) } ( F (\iota_Z^!  \sO_{Z_j} ) , F(\iota_Z^! \cc(x) ) ) \simeq  \\
\mathrm{  RHom}_{ \mathrm{  D}_Z^b (\mathrm{  Coh}(X)) } (\iota_Z^! \sO_{Z_j} , \iota_Z^! \cc(x)) \simeq \\
\mathrm{  RHom}_{ \mathrm{  D}^b (\mathrm{  Coh}(X)) } (\sO_{Z_j} , \cc(x) ) = 0
\end{align*} 
so that  $y\notin \mathrm{  Supp} ( F(\iota_Z^! \sO_{Z_j}) )$.
Set $S= \bigcup_{0\leq j<\dim Z} \mathrm{  Supp} ( F(\iota_Z^! \sO_{Z_j}) ) \subset W$. The set  $S$ is a closed subset and does not contain $y$.
Also define the following open dense subset of $W$: $$V:= W' \backslash S.$$

\begin{prop}\label{prop:tensorm}
For every $p\in V$ closed point  and $m\in \mathbb Z$ we have
$$\mathrm{  RHom} ( F ( \iota_Z^!  \sO_Z (m) ) , \iota_W^! \cc(p) ) \simeq \bigoplus_{k=0}^{\dim X} \mathbb C^{\oplus b^p_k}[-k].$$  
\end{prop}

\begin{proof}
 
For  every $j=1, \ldots , \dim Z$ there is a  short exact sequence:
$$0 \to \sO_{Z_{j}}(-Z_{j-1}) \to \sO_{Z_{j}} \to \sO_{Z_{j-1}} \to 0$$ where  $\sO_{Z_j}(-Z_{j-1})  \simeq  \sO_{Z_j}(-1) = \sO_{Z_j} \otimes 
\sO_{{\mathbb P}^N} (-1)$. 
We set $$\sO_{Z_j}(m) := \sO_{Z_j} \otimes \mathbb \sO_{{\mathbb P}^N}(m)$$ for  $m\in \mathbb Z$ and note the 
 distinguished triangles 
\begin{align*}
\iota_Z^! \sO_{Z_{j-1}} (m+1) [-1] \to \iota_Z^! \sO_{Z_j} (m) \to \iota_Z^! \sO_{Z_j} (m+1) \to \iota_Z^!  \sO_{Z_{j-1}} (m+1)\\
F( \iota_Z^! \sO_{Z_{j-1}} (m+1)  [-1] ) \to  F(\iota_Z^! \sO_{Z_j} (m) ) \to F( \iota_Z^! \sO_{Z_j} (m+1) ) 
\to F( \iota_Z^!  \sO_{Z_{j-1}} (m+1) )
\end{align*}
in $\mathrm{  D}^b_Z(\mathrm{   Coh}(X))$ and $\mathrm{  D}^b_W(\mathrm{   Coh}(Y))$, respectively.
Therefore we find that 
\begin{equation}\label{eq:mpos}
\mathrm{  Supp} ( F(\iota_Z^!  \sO_{Z_j}(m+1)) ) \subset \mathrm{  Supp} ( F(\iota_Z^!  \sO_{Z_j}(m)) ) \cup 
 \mathrm{  Supp} (F(\iota_Z^!  \sO_{Z_{j-1}}(m+1)) )
\end{equation} 
  and
\begin{equation}\label{eq:mneg}
 \mathrm{  Supp} ( F(\iota_Z^!  \sO_{Z_j}(m)) ) \subset \mathrm{  Supp} ( F(\iota_Z^!  \sO_{Z_j}(m+1))  )
 \cup \mathrm{  Supp} (F(\iota_Z^!  \sO_{Z_{j-1}}(m+1)) ) .
\end{equation}

Now note that $\iota_Z^! \sO_{Z_0}(m) \simeq \iota_Z^! \sO_{Z_0}$ for all $m$. 
By proceeding by induction on $m$, and by employing \eqref{eq:mpos} for $m\geq 0$ and \eqref{eq:mneg} for $m<0$, we conclude that 
\begin{align*}
\mathrm{  Supp} ( F(\iota_Z^!  \sO_{Z_{\dim Z-1}}(m)) ) \subset \\
 \mathrm{  Supp} ( F(\iota_Z^!  \sO_{Z_{\dim Z-1}}(m - 1)) ) \cup \mathrm{  Supp} ( F(\iota_Z^!  \sO_{Z_{\dim Z-2}}(m)) ) \subset  \\
\mathrm{  Supp} ( F(\iota_Z^!  \sO_{Z_{\dim Z-1}}(m-2))) \cup \mathrm{  Supp} ( F(\iota_Z^!  \sO_{Z_{\dim Z-2}}(m - 1)) )  \cup 
\mathrm{  Supp} ( F(\iota_Z^!  \sO_{Z_{\dim Z-3}}(m )) ) \\ \subset  \ldots \subset  
 \bigcup_{0\leq j<\dim Z} \mathrm{  Supp} ( F(\iota_Z^!  \sO_{Z_j})).       
\end{align*}

Finally  let $p\in V$ be any point closed point and 
consider the short exact  sequence $$0\to \sO_Z(-1) \to \sO_Z \to \sO_{Z_{\dim Z-1}} \to 0,$$ together with  the distinguished triangle
$$F(\iota_Z^! \sO_Z (-1) ) \to  F(\iota_Z^!  \sO_Z ) \to F( \iota_Z^! \sO_{Z_{\dim Z-1}} ) \to F(\iota_Z^! \sO_Z (-1) ) [1].$$
Note that 
$$\mathrm{  RHom} ( F( \iota_Z^!  \sO_{Z_{\dim Z-1}}) , \iota_W^! \cc(p)) = 0$$ as $p\notin S$. 
 Therefore we obtain the isomorphisms
$$\mathrm{  RHom} ( F( \iota_Z^!  \sO_{Z} (-1) ) , \iota_W^! \cc(p)) \simeq \mathrm{  RHom} ( F( \iota_Z^!   \sO_{Z}) , \iota_W^! \cc(p))
\simeq \bigoplus_{k=0}^{\dim X} \mathbb C^{\oplus b^p_k}[-k].$$ 
By running an induction argument on $m$ we get the statement. 
\end{proof}

\subsection{Step 3}
       
For every closed point $p\in V$  there exists an object $E_p$ in $\mathrm{  D}_Z^b(\mathrm{   Coh}(X))$ such that 
$F(E_p) = \iota_W^! \cc(p)$. Moreover, by Proposition \ref{prop:tensorm},   we have
$$\mathrm{  RHom} (\iota_Z^!\sO_Z(m)  , E_p  ) \simeq \mathrm{  RHom} ( F(\iota_Z^! \sO_Z(m) ) , \iota_W^! \cc(p)) \simeq 
\bigoplus_{k=0}^{\dim X} \mathbb C^{\oplus b^p_k}[-k]$$
for all $m$. Hence    by Proposition \ref{prop:charstr} $E_p$ is a point like object and 
 there is an isomorphism  $E_p \simeq \iota_Z^! \cc(q)[r_p]$ for some   $q\in Z(\cc)$ and $r_p\in \mathbb Z$. 

Define the   subset $$Z(\cc) \supset U(\cc) : = \{ q \; \big| \; F(\iota_Z^! \cc(q) ) =
 \iota_W^!\cc(p)[s_p]  \mbox{ for some }p \in V(\cc) \mbox{ and }s_p\in \mathbb Z \}$$
 and  the function $f\colon U(\cc) \to V(\cc)$  such that $f(q)=p$ if and only if $F(\iota_Z^!\cc(q)) = \iota_W^! \cc(p)[s_p]$. 
 Note that $f$ is injective and surjective.
  
\begin{prop}
The set $U(\cc)$ is open in $Z(\cc)$.
\end{prop}

\begin{proof}

Define the closed set  $C = W \backslash V$ and let $G_C$ be a classical generator of $\mathrm{  D}^b_C ( \mathrm{   Coh} (Y))$ (\cite[Theorem 6.8]{R}). 
Moreover let $F^{-1}$  be  a  quasi-inverse to $F$.
Then for any $q\in U(\cc)$ we have
$$\mathrm{  RHom} (F^{-1}(G_C) , \iota_Z^! \cc(q))   \simeq  \mathrm{  RHom} (G_C , F (\iota_Z^! \cc(q) ) )  \simeq  0$$
and therefore $\mathrm{  Supp} (F^{-1}(G_C) ) \cap Z(\cc) \subset Z(\cc)\backslash U(\cc)$.

Now we show the reverse inclusion. 
By \cite[Lemma 2.3]{LLZ},  the following inclusion holds for any object $E$ in $\mathrm{  D}^b_C ( \mathrm{   Coh}(Y))$:
\begin{equation}\label{eq:inclzu}
\mathrm{  Supp} ( F^{-1}(E) ) \cap Z(\cc)  \subset  \mathrm{  Supp} ( F^{-1}(G_C) ) \cap Z(\cc) .
\end{equation}
Let $q' \in Z(\cc) \backslash U(\cc)$ and  $p'\in \mathrm{  Supp} ( F ( \iota_Z^! \cc(q') ) ) \cap Z(\cc)$.
Then  Serre duality  yields
\begin{align*}
\mathrm{  RHom} ( F^{-1} ( \iota_W^! \cc(p') ) , \iota_Z^!  \cc(q') ) \simeq \\
 \mathrm{  RHom} (  \iota_W^! \cc(p') , F ( \iota_Z^!  \sO_{q'} ) ) \simeq \\
  \mathrm{  RHom} (  F ( \iota_Z^!  \cc(q') ), S_W (  \iota_W^! \cc(p') )   )^*
\simeq \\
 \mathrm{  RHom} (  F ( \iota_Z^!  \cc(q') ), \iota_W^! \cc(p')[\dim Y] )   )^*\neq 0.
\end{align*}
Hence $q' \in
\mathrm{  Supp} (F^{-1} (\iota_W^! \cc(p') ) ) \cap Z(\cc)$ and by \eqref{eq:inclzu}  it must be  $q'\in  \mathrm{  Supp} ( F^{-1} (G_C) ) \cap Z(\cc)$. 
 It follows that $Z(\cc)\backslash U(\cc) = \mathrm{  Supp} ( F^{-1} (G_C) ) \cap Z(\cc)$ 
is a closed set in $Z(\cc)$.
\end{proof}

\begin{prop}\label{prop:cont}
The functions $f$ and $f^{-1}$ are continuous.
\end{prop}

\begin{proof}
 
 If $V_0\subset V$ is a Zariski closed subset,  we denote by  $\overline{ V_0}$  its closure in $W$. 
Since for every $q\in U(\cc)$ we have 
$$\mathrm{  RHom}( F^{-1} (\iota_W^! \sO_{\overline{ V_0}} ) ,\iota_Z^! \cc(q))  
\;  \simeq \; \mathrm{  RHom}( \iota_W^! \sO_{\overline{ V_0}} , \iota_W^!\cc(f(q)) [s_{f(q)}]) ),$$ 
then 
$$q \in \mathrm{  Supp} ( F^{-1} (\iota_W^! \sO_{\overline{ V_0}}  ) ) \cap U(\cc) \;  \Leftrightarrow  \; f(q) \in \overline{ V_0}  \cap V(\cc)=  V_0 \cap V(\cc).$$ 
It follows that $f^{-1}(V_0 \cap V(\cc)) = \mathrm{  Supp} (  F^{-1} (\iota_W^! \sO_{\overline{ V_0}} )) \cap U $, which   is closed in $U$.

In a similar manner we prove that   $f^{-1}$ is continuous. 
For every Zariski closed subset $M$ in $U(\cc)$, we denote by $\overline M$ its  closure in $Z$.  Note that for every $q\in U(\cc)$ we have  
$$\mathrm{  RHom}( \iota_Z^! \sO_{\overline M} , \iota_Z^! \cc(q) ) \; \simeq  \; \mathrm{  RHom} (F( \iota_Z^! \sO_{\overline M} ) , 
\iota_W^! \cc(f(q)) [s_{f(q)}]).$$ 
Then it follows that 
$$f(q) \in \mathrm{  Supp}( F( \iota_Z^! \sO_{\overline M} ) ) \cap V (\cc) \; \Leftrightarrow \; q \in \overline M \cap U(\cc) = M$$
and
$f(M) = \mathrm{  Supp}  ( F( \iota_Z^! \sO_{\overline M} ) )  \cap V(\cc)$ is closed in $V(\cc)$.

\end{proof}

 \subsection{Step 4}
  \begin{cor}
 If $F$ is an equivalence and  $\dim W=1$, then $Z(\cc)$ and $W(\cc)$ are homeomorphic.
 \end{cor}
 
 \begin{proof}
Let $p\in W(\cc)\backslash V(\cc)$ and let $F_p$ be an object in $\mathrm{  D}^b_Z(\mathrm{  Coh}(X))$ such that $ F(F_p) = \iota_W^! \cc(p)$. The
support of $F_p$ is
properly contained in $Z$ as
$$ \mathrm{  RHom}(F_p , \iota_Z^! \cc(x)) \;  \simeq \;  \mathrm{  RHom} ( \iota_W^! \cc(p) , \iota_W^! \cc(y)) =0$$ (note that $y\in V(\cc)$). 
Hence $F_p$ is supported 
in dimension zero and, by Lemma \ref{lem:pointlike}, $F_p \simeq  \iota_Z^! \cc(q)[m_p]$ for some 
$q \in Z(\cc)\backslash U(\cc)$   and $m_p\in \mathbb Z$. Therefore we can extend $f$ on $q$ by setting $f(q)=p$. 
By iterating this process for every $p\in W(\cc) \backslash V(\cc)$, we obtain an injective and 
surjective function $f\colon U' \to W(\cc)$ where $U(\cc)\subset U' \subset Z(\cc)$. 

We show that $U'=Z(\cc)$. Suppose by contradiction  that there exists $r\in Z(\cc)\backslash U'$. Note that for every $q\in U'$:
$$0 \; = \;  \mathrm{  RHom} ( \iota_Z^! \cc(r) , \iota_Z^!\cc(q) )  \; \simeq \;
  \mathrm{  RHom} (F (\iota_Z^! \cc(r)) , F(\iota_Z^!\cc(q)) ) \; \simeq \;  \mathrm{ R Hom} (F (\iota_Z^! \cc(r)) , \iota_W^! \cc(f(q)) ).$$
As $f$ is surjective, this means that $\mathrm{  Supp} ( F(\iota_Z^! \cc(r)) )= \emptyset$ and $ F(\iota_Z^! \cc(r)) \simeq 0$ which is impossible.
Moreover, as in Proposition \ref{prop:cont}, one can prove that both extensions 
$f\colon Z(\cc) \to W(\cc)$ and $f^{-1}\colon W(\cc) \to Z(\cc)$ are continuous.

 \end{proof}

%
%
%
%
%

\bibliographystyle{amsalpha}
\bibliography{bibl}

\end{document}